\documentclass[11pt]{amsart}
\usepackage{amssymb}
\usepackage{amsmath}
\usepackage{amsfonts,latexsym,amstext}

\usepackage [T1]{fontenc}
\usepackage [latin1]{inputenc}
\usepackage{color}

\newtheorem{theorem}{Theorem}[section]
\newtheorem{definition}[theorem]{Definition}

\newtheorem{example}[theorem]{Example}
\newtheorem{lemma}[theorem]{Lemma}
\newtheorem{proposition}[theorem]{Proposition}
\newtheorem{corollary}[theorem]{Corollary}
\newtheorem{remark}[theorem]{Remark}

\begin{document}


\title{$M$-ideals of homogeneous polynomials}

\author{Ver\'{o}nica Dimant}

\thanks{Partially supported by ANPCyT PICT 05 17-33042. }

\address{Departamento de Matem\'{a}tica, Universidad de San
Andr\'{e}s, Vito Dumas 284, (B1644BID) Victoria, Buenos Aires,
Argentina and CONICET} \email{vero@udesa.edu.ar}


\begin{abstract}
We study the problem of whether $\mathcal{P}_w(^nE)$, the space of
$n$-homogeneous polynomials which are weakly continuous on bounded
sets, is an $M$-ideal in the space of continuous $n$-homogeneous
polynomials $\mathcal{P}(^nE)$. We obtain conditions that assure
this fact and present some examples. We prove that if
$\mathcal{P}_w(^nE)$ is an $M$-ideal in $\mathcal{P}(^nE)$, then
$\mathcal{P}_w(^nE)$ coincides with $\mathcal{P}_{w0}(^nE)$
($n$-homogeneous polynomials that are weakly continuous on bounded
sets at 0). We introduce a polynomial version of property $(M)$ and
derive that if $\mathcal{P}_w(^nE)=\mathcal{P}_{w0}(^nE)$ and
$\mathcal{K}(E)$ is an $M$-ideal in $\mathcal{L}(E)$, then
$\mathcal{P}_w(^nE)$ is an $M$-ideal in $\mathcal{P}(^nE)$. We also
show that if
$\mathcal{P}_w(^nE)$ is an $M$-ideal in $\mathcal{P}(^nE)$, then the
set of $n$-homogeneous polynomials whose Aron-Berner extension do
not attain the norm is nowhere dense in $\mathcal{P}(^nE)$. Finally,
we face an analogous $M$-ideal problem for block diagonal
polynomials.
\end{abstract}

\maketitle

\section*{Introduction}

In the theory of Banach spaces, the concept of $M$-ideal, since its
introduction by Alfsen and Effros \cite{AE}, proves to be an
important tool to study geometric and isometric properties of the
spaces. As it is quoted in the  exhaustive book on the topic written
by Harmand, Werner and Werner \cite{HWW}:

\begin{quotation}
``The fact that $Y$ is an $M$-ideal in $X$ has a strong impact on
both $Y$ and $X$ since there are a number of important properties
shared by $M$-ideals, but not by arbitrary subspaces.''
\end{quotation}

Being the space of compact operators between Banach spaces $X$ and
$Y$, $\mathcal{K}(X,Y)$, a distinguished subspace of the space of
linear operators $\mathcal{L}(X,Y)$, many specialists have been
interested in characterizing when $\mathcal{K}(X,Y)$ is an $M$-ideal
in $\mathcal{L}(X,Y)$  and deriving properties from this fact (see,
for example \cite{HWW,H,K,KW,L,OW}).

When changing to the polynomial setting, the role of ``compact
operators'' is usually played by the homogeneous polynomials which
are ``weakly continuous on bounded sets'' (that is, polynomials that send
bounded weak convergent nets into convergent nets). Recall that ``compact linear operators'' are the same as ``weakly  continuous on bounded sets linear operators''. For polynomials, these two concepts do not coincide, because continuous polynomials are not necessarily weak-to-weak continuous. Moreover, every scalar valued continuous polynomial is compact but not every scalar valued continuous polynomial is weakly continuous on bounded sets.
Many authors have studied
for which values of $n$, for a fixed Banach space $E$, all the
continuous $n$-homogeneous polynomials on $E$ are weakly continuous
on bounded sets (see, for instance \cite{AAD,BR-98,CGG,DD, Fa,Go,GJ,
Ry-79,Ry-80}). We are interested here, when this is not the case, to
determine if there is an $M$-structure.

\bigskip
Let us recall the definition and some facts that we use about
$M$-ideals. Naturally, our guide in this topic is the book
\cite{HWW}.

\begin{definition}
A closed subspace $J$ of a Banach space $X$ is an {\bf $M$-ideal} in
$X$ if
$$
X^*=J^\perp \oplus_1 J^\sharp,
$$

where $J^\perp$ is the annihilator of $J$ and $J^\sharp$ is a closed
subspace of $X^*$.
\end{definition}

When $J$ is an $M$-ideal in $X$, then
$$
J^\sharp=\{x^*\in X^*: \|x^*\|=\|x^*|_{J}\|\},
$$
and so $J^\sharp$ can be (isometrically) identified with $J^*$.
Thus, we usually denote $X^*=J^\perp \oplus_1 J^*$. Another relevant
result that we want to recall is that,  being $J$  an $M$-ideal in
$X$, we have the following equality about sets of extreme points of
 unit balls:
$$
Ext(B_{X^*})=Ext(B_{J^\perp})\cup Ext(B_{J^*}).
$$

Also, recall the 3-ball property \cite[Theorem I.2.2 (iv)]{HWW} that
we will use repeatedly: a closed subspace $J$ is an $M$-ideal in $X$
if and only if for every $y_1,y_2,y_3\in B_J$, $x\in B_X$ and
$\varepsilon >0$, there exists $y\in J$ satisfying
$$
\|x+y_i-y\|\leq 1+\varepsilon, \qquad i=1,2,3.
$$

\bigskip
Throughout this paper $E$ will denote a Banach space over
$\mathbb{K}$, where $\mathbb{K}=\mathbb{R}$ or $\mathbb{C}$. The closed unit ball of $E$ will be noted by $B_E$ and the unit sphere by $S_E$.
As
usual, $\mathcal{L}(E)$ and $\mathcal{K}(E)$ will be the notations
for the spaces of continuous linear and compact operators from $E$
to $E$.
 $\mathcal{P}(^nE)$ will denote the space of all continuous
$n$-homogeneous polynomials from $E$ to $\mathbb{K}$, which is a Banach space with the norm
$$
\|P\|=\sup\{|P(x)|:\, x\in B_E\}.
$$ If $P\in
\mathcal{P}(^nE)$, there exists a unique symmetric $n$-linear
mapping $\overset\vee P\colon\underbrace{E\times\cdots\times E}_n\to
\mathbb{K}$ such that
$$P(x)=\overset\vee P(x,\dots,x).$$

For a polynomial $P\in \mathcal{P}(^nE)$, its Aron-Berner extension
to $E^{**}$ will be noted by $\overline{P}$. Also, for each $z\in E^{**}$, $e_z$ will refer to the application given by $e_z(P)=\overline{P}(z).$

We will denote by $\mathcal{P}_{w}(^nE)$ the space of $n$-homogeneous
polynomials on $E$ that are weakly continuous on bounded sets (equivalently, that sent bounded weak convergent nets into convergent nets) and by
$\mathcal{P}_{w0}(^nE)$ the space of $n$-homogeneous polynomials on
$E$ that are weakly continuous on bounded sets at 0 (equivalently, that sent bounded weakly null nets into null nets).
Polynomials of the form
$P=\sum_{j=1}^N \pm  \gamma_j^n$, with $\gamma_j\in E^*$, are said to be of finite type. The space of finite type $n$-homogeneous polynomials  on $E$ will be denoted by
 $\mathcal{P}_f(^nE)$  and its closure (in the polynomial norm), which is the space of approximable $n$-homogeneous polynomials  on $E$ will be noted by $\mathcal{P}_A(^nE)$.

Recall that if $E$ does not contain a subspace isomorphic to
$\ell_1$, then $\mathcal{P}_w(^nE)$ coincides with the space of
weakly sequentially continuous $n$-homogeneous polynomials on $E$
 \cite{AHV}.

We refer to \cite{Di} for the necessary background on polynomials.

\bigskip
The paper is organized as follows. In section 1 we present some
general results and consequences of the fact that
$\mathcal{P}_w(^nE)$ is an $M$-ideal in $\mathcal{P}(^nE)$. The main
implication is that for this to happen it is necessary that
$\mathcal{P}_w(^nE)=\mathcal{P}_{w0}(^nE)$. We also obtain a
Bishop-Phelps type result for polynomials. In section 2 we state a
set of conditions that assure that $\mathcal{P}_w(^nE)$ is an
$M$-ideal in $\mathcal{P}(^nE)$. This enable us to produce some nice
examples of this fact. Section 3 is devoted to a polynomial version
of property $(M)$. With this property we can make a link with the
linear theory obtaining that if
$\mathcal{P}_w(^nE)=\mathcal{P}_{w0}(^nE)$ and $\mathcal{K}(E)$ is
an $M$-ideal in $\mathcal{L}(E)$, then $\mathcal{P}_w(^nE)$ is an
$M$-ideal in $\mathcal{P}(^nE)$. In the last section we consider the
space of block diagonal polynomials (with respect to a fixed
sequence of integers $J$) that are defined on a Banach space $E$
with an unconditional finite dimensional decomposition. We study in
this context the same question as before, that is whether the space of
weakly continuous on bounded sets block diagonal polynomials results an
$M$-ideal in the space of block diagonal polynomials.

\section{General results}

We begin by stating some easy results which are polynomial analogous
to \cite[Propositions VI.4.2 and VI.4.3]{HWW}. Their proofs are
straightforward.

First, recall that $J$ is an {\bf $M$-summand} in $X$ if
$$
X=J\oplus_\infty \widehat{J},
$$
where $\widehat{J}$ is a closed subspace of $X$. Clearly
$M$-summands are $M$-ideals.

\begin{proposition} \label{M-sumando}
If $\mathcal{P}_w(^nE)$ is an $M$-summand in $\mathcal{P}(^nE)$,
then $\mathcal{P}_w(^nE)=\mathcal{P}(^nE)$.
\end{proposition}

\begin{proposition} \begin{enumerate}
\item[(a)] If $\mathcal{P}_w(^nE)$ is an $M$-ideal in $\mathcal{P}(^nE)$ and $F\subset E$ is an 1-complemented subspace, then
$\mathcal{P}_w(^nF)$ is an $M$-ideal in $\mathcal{P}(^nF)$.

\item[(b)] The class of Banach spaces $E$ for which $\mathcal{P}_w(^nE)$ is an $M$-ideal in
$\mathcal{P}(^nE)$ is closed with respect to the Banach-Mazur
distance.
\end{enumerate}
\end{proposition}

One useful tool when computing norms in Banach spaces is the
characterization of the extreme points of the duals unit balls. Let
us observe some facts about extreme points when the space considered
is a space of polynomials.

First, note the following. If $J$ is a subspace of $\mathcal{P}(^nE)$  that contains $\mathcal{P}_f(^nE)$, then, for every $x\in S_E$, the  application $e_x$ belongs to $S_{J^*}$. Indeed, it is clear that $e_x(P)=P(x)$ is a linear functional on $J$ and that $\|e_x\|=\sup\{|e_x(P)|:\, P\in B_J\}\leq 1$. Also, since $J$ contains all finite type $n$-homogeneous polynomials, it contains $\gamma^n$, for every  $\gamma\in E^*$ and so $\|e_x\|=1$.

\begin{proposition}\label{extremales}
\begin{enumerate}
\item[(a)] If $J$ is a subspace of $\mathcal{P}(^nE)$  that contains all finite type $n$-homogeneous polynomials, then
$$
Ext B_{J^*}\subset \overline{\big\{\pm e_x
:x\in S_E\big\}}^{w^*},
$$where the $\pm$ is needed only in the real case and $w^*$ is the topology $\sigma(J^*,J)$.

%
%

\item[(b)]
For the particular case $J=\mathcal{P}_w(^nE)$ we can be more precise:
$$
Ext B_{\mathcal{P}_w(^nE)^*}\subset \{\pm e_z:z\in S_{E^{**}}\big\},
$$where the $\pm$ is needed only in the real case.
\end{enumerate}
\end{proposition}

\begin{proof}
(a) We have that
$$
B_{J^*}=\overline{\Gamma\{\pm e_x:x\in
S_E\}}^{w^*}.
$$ Indeed, one inclusion follows from the comment before the proposition and the other is easily obtained through Hahn-Banach theorem.

Now, Milman's theorem \cite[Theorem 3.41]{FHHMPZ} yields that
$$
Ext B_{J^*}\subset\overline{\{\pm e_x:x\in
S_E\}}^{w^*}.
$$

(b) When $J=\mathcal{P}_w(^nE)$, let us see that $\overline{\{\pm e_x:x\in
S_E\}}^{w^*}\subset\{\pm e_z:z\in S_{E^{**}}
\big\}$. If $\Phi\in\overline{\{\pm e_x:x\in
S_E\}}^{w^*}$, then there exists a net $\{x_\alpha\}_\alpha$ in $S_E$ such that $e_{x_\alpha}\overset{w^*}{\rightarrow}\Phi$ (or $-e_{x_\alpha}\overset{w^*}{\rightarrow}\Phi$).
In passing to appropriate subnets, we can suppose that $\{x_\alpha\}_\alpha$ is
$w^*$-convergent to an element $z$ in $S_{E^{**}}$ (here, $w^*$
means $\sigma(E^{**}, E^*)$). Since the Aron-Berner extension of a weakly continuous on bounded sets polynomial is $w^*$-continuous on bounded sets, we derive that $\overline{P}(x_\alpha)\to \overline{P}(z)$, for every $P\in \mathcal{P}_w(^nE)$. Thus, $e_{ x_\alpha}\overset{w^*}{\rightarrow}e_z$. Therefore, $\Phi=e_z$.
\end{proof}

Some observations are in order.

\begin{remark}\rm
\begin{enumerate}
\item For the particular case $J=\mathcal{P}(^nE)$, the previous result can also be proved using the representation of $\mathcal{P}(^nE)$ as the dual of $\widehat{\otimes}^n_{\pi_s}E$ (the symmetric projective $n$-tensor product of $E$), the description of the unit ball of this space given in \cite{Flo} and Goldstine's theorem.

\item When $E^*$ has the approximation property, item (b) of the proposition above can also be obtained from another argument. Indeed, in this case, we have the equality
$\mathcal{P}_w(^nE)=\mathcal{P}_A(^nE)$ and then its dual is the space of
integral $n$-homogeneous polynomials on $E^*$:
$\mathcal{P}_w(^nE)^*=\mathcal{P}_I(^nE^*)$. By the description of
the set of extreme points of the ball of the space of integral
polynomials in  \cite[Proposition 1]{BR-01} (see also \cite[Proof of
Theorem 1.5]{CD}), the result is obtained.
\end{enumerate}
\end{remark}

The essential norm of a linear operator is the distance to the
subspace of compact operators. Analogously, we define:

\begin{definition}
For an $n$-homogeneous polynomial $P\in\mathcal{P}(^nE)$, the {\bf
essential norm} of $P$ is
$$
\|P\|_{es}=d(P,\mathcal{P}_w(^nE))=\inf\{\|P-Q\|: Q\in
\mathcal{P}_w(^nE)\}.
$$
\end{definition}

The following result, which is the polynomial version of
\cite[Proposition VI.4.7]{HWW} has an important consequence stated
in the corollary below.

\begin{proposition}
For any $P\in\mathcal{P}(^nE)$, consider
$$
w(P)=\sup\big\{\limsup|P(x_\alpha)|: \|x_\alpha\|=1,
x_\alpha\overset{w}{\rightarrow} 0\big\}.
$$
If $\mathcal{P}_w(^nE)$ is an $M$-ideal in $\mathcal{P}(^nE)$, then
$\|P\|_{es}=w(P)$.
\end{proposition}

\begin{proof}
Let $Q\in\mathcal{P}_w(^nE)$ and a weakly null net
$\{x_\alpha\}_\alpha$ with $\|x_\alpha\|=1$, for all $\alpha$. Then,
$$
\|P-Q\|\geq\big|(P-Q)(x_\alpha)\big|\geq \big|P(x_\alpha)\big|-
\big|Q(x_\alpha)\big|.
$$
Since $Q(x_\alpha)\to 0$ it follows that
$\|P-Q\|\geq\limsup|P(x_\alpha)|$ and thus $\|P\|_{es}\geq w(P)$.

Now suppose that $\mathcal{P}_w(^nE)$ is an $M$-ideal in
$\mathcal{P}(^nE)$. Then we have
$$
\mathcal{P}(^nE)^*=\mathcal{P}_w(^nE)^\perp \oplus_1
\mathcal{P}_w(^nE)^*\qquad\textrm{ and }\qquad
ExtB_{\mathcal{P}(^nE)^*}=ExtB_{\mathcal{P}_w(^nE)^\perp}\cup
ExtB_{\mathcal{P}_w(^nE)^*}.
$$

The essential norm of $P$, $\|P\|_{es}$, is the norm of the class of
$P$ in the quotient space $\mathcal{P}(^nE)/\mathcal{P}_w(^nE)$ and
the dual of this quotient can be isometrically identified with
$\mathcal{P}_w(^nE)^\perp$. Then, there exists $\Phi\in Ext
B_{\mathcal{P}_w(^nE)^\perp}$ such that $\Phi(P)=\|P\|_{es}$. So,
$\Phi\in Ext B_{\mathcal{P}(^nE)^*}$ and, by Proposition \ref{extremales} (a),
$\Phi\in \overline{\big\{\pm e_x:x\in S_E\big\}}^{w^*}$.

Then, there exists a net $\{x_\alpha\}_\alpha$ in $S_E$ such that
$e_{x_\alpha}\overset{w^*}{\rightarrow}\Phi$ (or $-e_{x_\alpha}\overset{w^*}{\rightarrow}\Phi$), where $w^*$ means the
topology $\sigma(\mathcal{P}(^nE)^*,\mathcal{P}(^nE))$. In passing
to appropriate subnets, we can suppose that $\{x_\alpha\}_\alpha$ is
$w^*$-convergent to an element $z$ in $S_{E^{**}}$ (here, $w^*$
means $\sigma(E^{**}, E^*)$).

If $\gamma\in E^*$, since
$\gamma^n\in\mathcal{P}_w(^nE)$ we obtain
$$ 0=\Phi(\gamma^n)=\lim_\alpha \gamma(x_\alpha)^n=
z(\gamma)^n.$$

So we have proved that $z(\gamma)=0$, for all $\gamma\in E^*$, which
implies that $z=0$ and that $\{x_\alpha\}_\alpha$ is weakly null. As
a consequence,
$$
\|P\|_{es}=\Phi(P)=\lim_\alpha P(x_\alpha)\leq w(P),
$$ which completes the proof.

\end{proof}

Since for every polynomial $P$ which is weakly continuous on bounded
sets at 0 it holds that $w(P)=0$ and the equality $\|P\|_{es}=0$
implies that $P$ is weakly continuous on bounded sets, we obtain:

\begin{corollary} \label{w-w0}
If $\mathcal{P}_w(^nE)$ is an $M$-ideal in $\mathcal{P}(^nE)$, then
$\mathcal{P}_w(^nE)=\mathcal{P}_{w0}(^nE)$.
\end{corollary}

This corollary tells us that we have at most one value of $n$ (which
we call ``critical degree'') where $\mathcal{P}_w(^nE)$ could be a
non trivial $M$-ideal in $\mathcal{P}(^nE)$. Indeed, recall the
following simple facts, whose proofs appeared in (or can be derived
from) the articles \cite{BR-98,AD}:
\begin{itemize}
\item If an $n$-homogeneous polynomial $P$ is weakly continuous on
bounded sets at any point $x\in E$, then it is weakly continuous on
bounded sets at 0.
\item If $\mathcal{P}_w(^kE)=\mathcal{P}(^kE)$, for all $1\leq k<n$, then
$\mathcal{P}_w(^nE)=\mathcal{P}_{w0}(^nE)$.
\item If there exists an $n$-homogeneous polynomial $P$ which is not
weakly continuous on bounded sets at a point $x\in E$, $x\not=0$,
then, being $\gamma\in E^*$, such that $\gamma(x)\not=0$, the
$(n+k)$-homogenous polynomial $Q=\gamma^kP$ belongs to
$\mathcal{P}_{w0}(^{n+k}E)\setminus\mathcal{P}_{w}(^{n+k}E)$.
\end{itemize}
So, the situation can be summarized as follows:

\begin{remark}\label{n-unico}\rm
For a Banach space $E$, either $\mathcal{P}_w(^kE)=\mathcal{P}_{w0}(^kE)=\mathcal{P}(^kE)$,
for all $k$, or there exists $n\in\mathbb{N}$ such that:
\begin{itemize}
\item $\mathcal{P}_w(^kE)=\mathcal{P}_{w0}(^kE)=\mathcal{P}(^kE)$,
for all $k< n$.
\item $\mathcal{P}_w(^nE)=\mathcal{P}_{w0}(^nE)\subsetneqq
\mathcal{P}(^nE)$.
\item $\mathcal{P}_w(^kE)\subsetneqq \mathcal{P}_{w0}(^kE)\subset \mathcal{P}(^kE)$,
for all $k> n$.
\end{itemize}
When this value of $n$ does exist, we call it the critical degree of $E$ and denote $n=cd(E)$.
\end{remark}

Therefore, if there exists a polynomial on $E$ which is not weakly
continuous on bounded sets, the critical degree is the minimum of
all $k$ such that $\mathcal{P}_{w}(^kE)\not=\mathcal{P}(^kE)$.


\begin{remark}\label{refle}\rm
If $E$ is a reflexive Banach space with the approximation property,
then $\mathcal{P}_{w}(^kE)^{**}=\mathcal{P}(^kE)$, for every $k$.
When this is the case, the problem that we are studying here is
whether $\mathcal{P}_{w}(^nE)$ is an $M$-ideal in its bidual.
\end{remark}

\bigskip
We finish this section by a polynomial version of \cite[Proposition
VI.4.8]{HWW}. This result can be related with the extensions of the
Bishop-Phelps theorem to the polynomial setting. There is no
polynomial (nor multilinear) Bishop-Phelps theorem \cite{AAP, Ac},
but there are some variations that are valid. Aron, Garc\'{\i}a and
Maestre \cite{AGM} proved that the set of 2-homogeneous polynomials
whose Aron-Berner extension attain their norm is dense in the set of
2-homogeneous polynomials. It is an open problem if this result can
be generalized for $n$-homogeneous polynomials. Here, with the very
strong hypothesis of $\mathcal{P}_w(^nE)$ being an $M$-ideal in
$\mathcal{P}(^nE)$, we obtain a stronger conclusion.

\begin{proposition}
Let $E$ be a Banach space and suppose that $\mathcal{P}_w(^nE)$ is an $M$-ideal in
$\mathcal{P}(^nE)$.
\begin{enumerate}
\item[(a)] If $P\in \mathcal{P}(^nE)$ is such that its Aron-Berner extension $\overline{P}$
does not attain its norm at $B_{E^{**}}$, then $\|P\|=\|P\|_{es}$.
\item[(b)] The set of polynomials in $\mathcal{P}(^nE)$ whose
Aron-Berner extension do not attain the norm is nowhere dense in
$\mathcal{P}(^nE)$.
\end{enumerate}
\end{proposition}

\begin{proof}
(a) Let $\Phi\in Ext B_{\mathcal{P}(^nE)^*}$ such that
$\|P\|=\Phi(P)$. Being $\mathcal{P}_w(^nE)$ an $M$-ideal in
$\mathcal{P}(^nE)$, it holds that $\Phi\in Ext
B_{\mathcal{P}_w(^nE)^*}$ or $\Phi\in Ext
B_{\mathcal{P}_w(^nE)^\perp}$. If $\Phi\in Ext
B_{\mathcal{P}_w(^nE)^*}$ then, by Proposition \ref{extremales} (b),
$\Phi=\pm e_z$, with $z\in E^{**}$, $\|z\|=1$. Then,
$$
\|\overline{P}\|=\|P\|=\Phi(P)=|\overline{P}(z)|,
$$
which is a contradiction. Thus, $\Phi\in Ext
B_{\mathcal{P}_w(^nE)^\perp}$. Consequently,
$$
\|P\|=\Phi(P)=\sup\big\{|\Psi(P)|: \Psi\in Ext
B_{\mathcal{P}_w(^nE)^\perp}\big\}=\|P\|_{es}.
$$

(b) By (a), the set of polynomials in $\mathcal{P}(^nE)$ whose
Aron-Berner extension do not attain the norm is contained in the
metric complement
$$
\mathcal{P}_w(^nE)^{\Theta}=\big\{P\in \mathcal{P}(^nE):
\|P\|=\|P\|_{es}\big\}.
$$
Since this set is closed, we have to prove that it has empty
interior. By \cite[Proposition II.1.11 and Corollary II.1.7]{HWW},
$\mathcal{P}_w(^nE)^{\Theta}$ has empty interior if and only if
$$
\inf\left\{\sup_{\Phi\in  B_{\mathcal{P}_w(^nE)^*}}|\langle\Phi,
P\rangle|:\ \|P\|_{es}=1\right\}=1.
$$
But we have
\begin{eqnarray*}
\sup_{\Phi\in  B_{\mathcal{P}_w(^nE)^*}}|\langle\Phi, P\rangle|
&\leq & \sup_{\Phi\in  B_{\mathcal{P}(^nE)^*}}|\langle\Phi,
P\rangle| = \|P\| = \sup_{x\in B_E}|P(x)|\\
& =& \sup_{x\in B_E} |\langle e_x,P\rangle| \leq \sup_{\Phi\in
B_{\mathcal{P}_w(^nE)^*}}|\langle\Phi, P\rangle|.
\end{eqnarray*}
Thus,
$$
\inf\left\{\sup_{\Phi\in  B_{\mathcal{P}_w(^nE)^*}}|\langle\Phi,
P\rangle|:\ \|P\|_{es}=1\right\}= \inf\Big\{ \|P\|:
\|P\|_{es}=1\Big\}=1,
$$ since $\|P\|_{es}\leq \|P\|$, for all $P$, and for $P\in
\mathcal{P}_w(^nE)^{\Theta}$, they coincide.
\end{proof}

\section{Compact approximations}

Several criteria for $\mathcal{K}(X,Y)$ to be an $M$-ideal in
$\mathcal{L}(X,Y)$ were related to the so-called ``shrinking compact
approximations of the identity'' satisfying certain properties. In
order to obtain examples of Banach spaces $E$ such that
$\mathcal{P}_w(^nE)$ is an $M$-ideal in $\mathcal{P}(^nE)$, we
present here a sufficient condition for this to happen, also
involving nets of compact operators. Due to Corollary \ref{w-w0} and Remark \ref{n-unico}, we
have to add the hypothesis of $n=cd(E)$. For proving the result
we need the following simple lemma.

\begin{lemma}\label{Pw}
Let $E$ be a Banach space and suppose that there exists a bounded
net $\{S_\alpha\}_\alpha$ of linear operators from $E$ to $E$
satisfying $S_\alpha^* \gamma \to \gamma$, for all $\gamma\in E^*$.
Then, for all $P\in\mathcal{P}_w(^nE)$, we have that $\|P-P\circ
S_\alpha\|\to 0$.
\end{lemma}

\begin{proof}
Being $\{S_\alpha^*\}_\alpha$ a bounded net it follows that
$S_\alpha^* \gamma \to \gamma$, uniformly for $\gamma$ in a
relatively compact set. Thus, for every Banach space $F$ and every
compact operator $K\in\mathcal{K}(E,F)$, it is clear that
$$
\|K-K\circ S_\alpha\|\to 0.
$$
It is known \cite{AHV} that for $P\in\mathcal{P}_w(^nE)$, its
associated linear operator
 $T_P$ belongs to $\mathcal{K}(E,\mathcal{L}_s(^{n-1}E))$ (where $\mathcal{L}_s(^{n-1}E)$ denote
 the space of symmetric $(n-1)$-linear forms on $E$). If $C$ is a bound for $\|S_\alpha\|$, we obtain
that
\begin{eqnarray*}
|P(x)-P\circ S_\alpha(x)| &=& \Big|\sum_{j=1}^n \binom{n}{j}
P^{\vee} \Big((x-S_\alpha(x))^j,S_\alpha(x)^{n-j}\Big)\Big|\\
&\leq & \sum_{j=1}^n \binom{n}{j} \Big|T_P(x-S_\alpha(x)) \Big((x-S_\alpha(x))^{j-1},S_\alpha(x)^{n-j}\Big)\Big|\\
&\leq & \sum_{j=1}^n \binom{n}{j} \|T_P-T_P\circ S_\alpha\| \|I -
S_\alpha\|^{j-1} \|S_\alpha\|^{n-j} \|x\|^n\\
&\leq & \sum_{j=1}^n \binom{n}{j} \|T_P-T_P\circ S_\alpha\|
(1+C)^{j-1} C^{n-j} \|x\|^n.
\end{eqnarray*}

Consequently, $\|P-P\circ S_\alpha\|\to 0$.

\end{proof}

\begin{proposition}\label{condition}
Let $E$ be a Banach space and suppose that there exists a bounded
net $\{K_\alpha\}_\alpha$ of compact operators from $E$ to $E$
satisfying the following two conditions:
\begin{itemize}
\item $K_\alpha^* \gamma \to \gamma$, for all $\gamma\in E^*$.
\item For all $\varepsilon >0$ and all $\alpha_0$ there exists $\alpha >\alpha_0$ such that for every $x\in E$,
$$
\|K_\alpha x\|^n + \|x-K_\alpha x\|^n \leq (1+\varepsilon)
\|x\|^n.
$$
\end{itemize}
Suppose also that $n=cd(E)$, then
$\mathcal{P}_w(^nE)$ is an $M$-ideal in $\mathcal{P}(^nE)$.
\end{proposition}

\begin{proof}
 Let $P_1, P_2, P_3\in B_{\mathcal{P}_w(^nE)}$, $Q\in
B_{\mathcal{P}(^nE)}$ and $\varepsilon >0$. In order to verify the
3-ball property \cite[Theorem I.2.2]{HWW}, we have to show that
there exists a polynomial $P\in \mathcal{P}_w(^nE)$ such that
\begin{equation}\label{3-ball}
\|Q+P_i-P\|\leq 1+ \varepsilon,\qquad \textrm{ for }i=1,2,3.
\end{equation}
Since, by the previous lemma, $\|P_i-P_i\circ K_\alpha\|\to 0$ for
$i=1,2,3$, let us fixed a value of $\alpha$ such that
$$
\|P_i-P_i\circ K_\alpha\|\leq \frac{\varepsilon}{2},\qquad \textrm{
for } i=1,2,3
$$
and also
$$
\|K_\alpha x\|^n + \|x-K_\alpha x\|^n \leq
\Big(1+\frac{\varepsilon}{2}\Big) \|x\|^n.
$$
Consider the polynomial $P\in \mathcal{P}(^nE)$ given by
$$
P(x)=Q(x)-Q(x-K_\alpha x).
$$
We have to prove that $P$ is weakly continuous on bounded sets and
that satisfies inequality (\ref{3-ball}).

Let $\{x_\beta\}_\beta$ be a bounded weakly null net. If we show
that $P(x_\beta)\to 0$, then $P$ should be weakly continuous on
bounded sets (since $\mathcal{P}_w(^nE)=\mathcal{P}_{w0}(^nE)$). By
the compacity of $K_\alpha$, we have that $K_\alpha x_\beta \to 0$,
as $\beta\to\infty$, and so
\begin{eqnarray*}
|P(x_\beta)|&=& \big|Q(x_\beta)-Q(x_\beta-K_\alpha x_\beta)\big|
\leq \sum_{j=1}^n \binom{n}{j} \left|Q^\vee \Big( (K_\alpha
x_\beta)^j,
(x_\beta - K_\alpha x_\beta)^{n-j}\Big)\right|\\
&\leq & \sum_{j=1}^n \binom{n}{j} \|Q^\vee\| \big\|K_\alpha
x_\beta\big\|^j \big\|x_\beta - K_\alpha x_\beta\big\|^{n-j}\leq
\sum_{j=1}^n \binom{n}{j} \|Q^\vee\| \big\|K_\alpha x_\beta\big\|^j
\big( (1+C_1)C_2\big)^{n-j}\underset{\beta}{\longrightarrow} 0,
\end{eqnarray*}
where $C_1$ and $C_2$ are bounds for the nets $\{K_\alpha\}_\alpha$
and $\{x_\beta\}_\beta$, respectively.

Equation (\ref{3-ball}) yields from the inequalities
$$
\|Q+P_i-P\|\leq \|Q+P_i\circ K_\alpha-P\| + \|P_i-P_i\circ
K_\alpha\|\leq \|Q+P_i\circ K_\alpha-P\| + \frac{\varepsilon}{2},
$$
and
\begin{eqnarray*}
\|Q+P_i\circ K_\alpha-P\| &=& \sup_{x\in B_E} \big|Q(x-K_\alpha x) +
P_i(K_\alpha x)\big| \leq \sup_{x\in B_E} \|K_\alpha x\|^n +
\|x-K_\alpha x\|^n\\
&\leq & \sup_{x\in B_E} \Big(1+\frac{\varepsilon}{2}\Big) \|x\|^n
= 1+\frac{\varepsilon}{2}.
\end{eqnarray*}
\end{proof}

\begin{remark}\rm
In \cite{OW}, Oja and Werner introduced the concept of $(M_p)$-space
as a space $X$ such that $\mathcal{K}(X\oplus_p X)$ is an $M$-ideal
in $\mathcal{L}(X\oplus_p X)$. By \cite[Theorem VI.5.3]{HWW}, for
$p\leq n$, every $(M_p)$-space fulfils the conditions (about the
net of compact operators) of Proposition \ref{condition}.
\end{remark}

For spaces with shrinking finite dimensional decomposition we have
the following simpler version of Proposition \ref{condition}.

\begin{corollary}\label{fdd}
Let $E$ be a Banach space with a shrinking finite dimensional
decomposition with associate projections $\{\pi_m\}_m$ such that:
\begin{itemize}
\item For all $\varepsilon >0$ and all $m_0\in\mathbb{N}$ there
exists $m>m_0$ such that for every $x\in E$,
$$
\|\pi_m x\|^n + \|x-\pi_m x\|^n \leq (1+\varepsilon) \|x\|^n.
$$
\end{itemize}
Suppose also that $n=cd(E)$. Then,
$\mathcal{P}_w(^nE)$ is an $M$-ideal in $\mathcal{P}(^nE)$.
\end{corollary}


From Proposition \ref{condition}, Corollary \ref{fdd} and Remark
\ref{n-unico}, we can derive the following examples.

\begin{example}\rm
If $H$ is a Hilbert space, then
$\mathcal{P}(^2H)\not=\mathcal{P}_{w}(^2H)$, because if
$\{e_\alpha\}_\alpha$ is an orthonormal basis then the polynomial
$$
P(x)=\sum_{\alpha}\langle x,e_\alpha\rangle^2
$$
is not weakly continuous on bounded sets. So the critical degree is
$n=2$ and since it is clear that the hypothesis of Proposition
\ref{condition} are valid, we get that $\mathcal{P}_w(^2H)$ is an
$M$-ideal in $\mathcal{P}(^2H)$.
\end{example}

\begin{example} \rm
Let $E=\bigoplus_{\ell_p}X_m$, where each $X_m$ is a finite
dimensional space and $1<p<\infty$. Then
$\mathcal{P}(^kE)=\mathcal{P}_{w}(^kE)$ if and only if $k<p$. This
means that the critical degree is the number $n$ that verifies
$p\leq n<p+1$. It is obvious that $E$ satisfies the hypothesis of
Corollary \ref{fdd}, and thus $\mathcal{P}_w(^nE)$ is an $M$-ideal
in $\mathcal{P}(^nE)$.

\noindent In particular, we have the result for $\ell_p$ spaces:
$\mathcal{P}_w(^n\ell_p)$ is an $M$-ideal in
$\mathcal{P}(^n\ell_p)$, for $p\leq n<p+1$.
\end{example}

\begin{example} \rm
Let us consider a dual of a Lorentz sequence space $E=d^*(w,p)$,
with $1<p<\infty$. Our result would work for certain sequences $w$.
If $n-1$ is the greatest integer strictly smaller than $p^*$,
suppose that $w\not\in\ell_s$, where
$s=\Big(\frac{(n-1)^*}{p}\Big)^*$. Then, by \cite[Proposition
2.4]{JP}, $n=cd(E)$. We obtain that $\mathcal{P}_w(^nE)$ is an $M$-ideal in
$\mathcal{P}(^nE)$. Indeed, $d^*(w,p)$ has a shrinking Schauder
basis $\{e_j\}_j$ and if $\pi_m$ is the projection
$\pi_m(x)=\sum_{j=1}^m x_je_j$, it holds that
$$
\|\pi_m x\|^n + \|x-\pi_m x\|^n \leq \Big(\|\pi_m x\|^{p^*} +
\|x-\pi_m x\|^{p^*}\Big)^{\frac{n}{p^*}} \leq  \|x\|^n.
$$
The last inequality follows by duality, since if $y$ and $z$ are
disjointly supported vectors  in $d(w,p)$, it holds that
$$
\|y+z\|\leq \Big(\|x\|^p+\|y\|^p\Big)^{\frac{1}{p}}.
$$
Observe that, in particular, we have proved that for $p\geq 2$,
$\mathcal{P}_w(^2d^*(w,p))$ is an $M$-ideal in
$\mathcal{P}(^2d^*(w,p))$, for any sequence $w$ (the above condition
in this case is $w\not\in\ell_1$, which is implied by the definition
of $d^*(w,p)$).
\end{example}

\begin{example}\label{Lp}\rm
Let $1<p<2$ and consider the space $L_p[0,1]$. Since $L_p[0,1]$
contains a complemented subspace isomorphic to $\ell_2$, it follows that
$\mathcal{P}(^2L_p[0,1])\not=\mathcal{P}_{w}(^2L_p[0,1])$ and $n=2$
is the critical degree. Even though we will see in Example
\ref{Lp-no-es} that $\mathcal{P}_{w}(^2L_p[0,1])$ is not an
$M$-ideal in $\mathcal{P}(^2L_p[0,1])$, the space $L_p[0,1]$ can be
renormed to a Banach space $E$ such that $\mathcal{P}_{w}(^2E)$ is
an $M$-ideal in $\mathcal{P}(^2E)$. Indeed, the renorming considered
in \cite[Proposition 6.8]{HWW} verifies all the conditions of
Corollary \ref{fdd}.
\end{example}

\begin{remark}\rm
The spaces $E$ of the previous examples are all reflexive with the
approximation property. So, by Remark \ref{refle}, the corresponding
spaces $\mathcal{P}_{w}(^nE)$ are $M$-embedded.
\end{remark}

\section{Polynomial property $(M)$}

\begin{lemma}
If $\mathcal{P}_w(^nE)$ is an $M$-ideal in $\mathcal{P}(^nE)$ then,
for each $P\in \mathcal{P}(^nE)$ there exists a bounded net
$\{P_\alpha\}_\alpha\subset \mathcal{P}_w(^nE)$ such that
$\overline{P}_\alpha(z)\to \overline{P}(z)$, for all $z\in E^{**}$.
\end{lemma}

\begin{proof}
By \cite[Remark I.1.13]{HWW}, if $\mathcal{P}_w(^nE)$ is an
$M$-ideal in $\mathcal{P}(^nE)$ then $B_{\mathcal{P}_w(^nE)}$ is
$\sigma\Big(\mathcal{P}(^nE),\mathcal{P}_w(^nE)^*\Big)$-dense in
$B_{\mathcal{P}(^nE)}$. So, for each $P\in B_{\mathcal{P}(^nE)}$
there exists a net $\{P_\alpha\}_\alpha\in B_{\mathcal{P}_w(^nE)}$
such that $P_\alpha\to P$ in the topology
$\sigma\Big(\mathcal{P}(^nE),\mathcal{P}_w(^nE)^*\Big)$.

Note that $\mathcal{P}_w(^nE)^*$ can be seen inside
$\mathcal{P}(^nE)^*$ by the identification with the set
$$
\mathcal{P}_w(^nE)^\sharp=\{\Phi\in \mathcal{P}(^nE)^*:\
\|\Phi\|=\|\Phi|_{\mathcal{P}_w(^nE)}\|\}.
$$
So, since $\|e_z\|=\|e_z|_{\mathcal{P}_w(^nE)}\|$, this implies that $\overline{P}_\alpha(z)\to
\overline{P}(z)$, for all $z\in E^{**}$.
\end{proof}

As a consequence of \cite[Proposition 2.3]{W} and the previous
lemma, we have the following result which can be proved analogously
to \cite[Theorem 3.1]{W}:

\begin{theorem}\label{equivalencias}
Let $E$ be a Banach space. The following are equivalent:
\begin{enumerate}
\item[(i)] $\mathcal{P}_w(^nE)$ is an
$M$-ideal in $\mathcal{P}(^nE)$.
\item[(ii)] For all $P\in \mathcal{P}(^nE)$ there exists a net
$\{P_\alpha\}_\alpha\subset \mathcal{P}_w(^nE)$ such that
$\overline{P}_\alpha(z)\to \overline{P}(z)$, for all $z\in E^{**}$
and
$$
\limsup \|Q+P-P_\alpha\|\leq \max\{\|Q\|, \|Q\|_{es} + \|P\|\},\quad
\textrm{for all } Q\in\mathcal{P}(^nE).
$$
\item[(iii)] For all $P\in \mathcal{P}(^nE)$ there exists a net
$\{P_\alpha\}_\alpha\subset \mathcal{P}_w(^nE)$ such that
$\overline{P}_\alpha(z)\to \overline{P}(z)$, for all $z\in E^{**}$
and
$$
\limsup_\alpha \|Q+P-P_\alpha\|\leq \max\{\|Q\|, \|P\|\},\quad
\textrm{for all } Q\in\mathcal{P}_w(^nE).
$$
\end{enumerate}
\end{theorem}

The property $(M)$, introduced by Kalton in \cite{K}, proves to be
useful to characterize the spaces $X$ such that $\mathcal{K}(X)$ is
an $M$-ideal in $\mathcal{L}(X)$. Recall one of its equivalent
definitions \cite[Lemma VI.4.13]{HWW}:

\begin{definition}
A Banach space $E$ has {\bf property $(M)$} if whenever $u,v\in E$
with $\|u\|\leq\|v\|$ and $\{x_\alpha\}_\alpha\subset E$ is a bounded weakly null net, then
$$
\limsup_\alpha \|u+x_\alpha\|\leq \limsup_\alpha \|v+x_\alpha\|.
$$
\end{definition}

In \cite{KW} an operator version of this property is introduced to
study when $\mathcal{K}(X,Y)$ is an $M$-ideal in $\mathcal{L}(X,Y)$.
Here we propose a polynomial version of property $(M)$.

\begin{definition}
Let $P\in\mathcal{P}(^nE)$ with $\|P\|\leq 1$. We say that $P$ has
{\bf property $(M)$} if for all $\lambda\in \mathbb{K}$, $v\in E$
with $|\lambda|\leq\|v\|^n$ and for every bounded weakly null net
$\{x_\alpha\}_\alpha\subset E$, it holds that
$$
\limsup_\alpha |\lambda+Px_\alpha|\leq \limsup_\alpha
\|v+x_\alpha\|^n.
$$
\end{definition}

Analogously to \cite[Lemma 6.2]{KW}, we can prove:

\begin{lemma}\label{redes}
Let $P\in\mathcal{P}(^nE)$ with $\|P\|\leq 1$. If $P$ has property
$(M)$ then for every net $\{v_\alpha\}_\alpha$ contained in a
compact set of $E$, for every net $\{\lambda_\alpha\}_\alpha\subset \mathbb{K}$
with $|\lambda_\alpha|\leq\|v_\alpha\|^n$ and for every bounded
weakly null net $\{x_\alpha\}_\alpha\subset E$, it holds that
$$
\limsup_\alpha |\lambda_\alpha+Px_\alpha|\leq \limsup_\alpha
\|v_\alpha+x_\alpha\|^n.
$$
\end{lemma}

\begin{definition}
We say that a Banach space $E$ has the {\bf $n$-polynomial property
$(M)$} if every $P\in\mathcal{P}(^nE)$ with $\|P\|\leq 1$ has
property $(M)$.
\end{definition}

The following proposition and theorem are the polynomial versions of
\cite[Theorem 6.3]{KW} and their proofs follow the ideas of the
proof of that theorem, with the necessary changes to the polynomial
setting.

\begin{proposition}
If $\mathcal{P}_w(^nE)$ is an $M$-ideal in $\mathcal{P}(^nE)$ then
$E$ has the $n$-polynomial property $(M)$.
\end{proposition}

\begin{proof}
Let $P\in\mathcal{P}(^nE)$ with $\|P\|\leq 1$, let $\lambda\in
\mathbb{K}$ and $v\in E$ with $|\lambda|\leq\|v\|^n$. Consider a bounded weakly null net
$\{x_\alpha\}_\alpha\subset E$. Take
$Q\in \mathcal{P}_w(^nE)$ such that $\|Q\|\leq 1$ and
$Q(v)=\lambda$. Given $\varepsilon >0$, by Theorem
\ref{equivalencias}(iii), there exists a polynomial
$R\in\mathcal{P}_w(^nE)$ such that
$$
|Pv-Rv|<\varepsilon\qquad \textrm{ and }\qquad \|Q+P-R\|\leq
1+\varepsilon.
$$
Since $Q(v+x_\alpha)\to Q(v)$ and $R(x_\alpha)\to 0$ we obtain:
\begin{eqnarray*}
\limsup_\alpha |\lambda+Px_\alpha| &=& \limsup_\alpha
|Q(v)+Px_\alpha|= \limsup_\alpha |Q(v+x_\alpha)+(P-R)x_\alpha|\\
&\leq & \limsup_\alpha |Q(v+x_\alpha)+(P-R)x_\alpha +
(P-R)(v)|+\varepsilon.
\end{eqnarray*}
Recall that the fact that $\mathcal{P}_w(^nE)$ is an $M$-ideal in
$\mathcal{P}(^nE)$, implies that for every $1\leq k\leq n-1$, all the polynomials
in $\mathcal{P}(^kE)$ are weakly continuous on bounded sets. So we
derive that
$$
\left|(P-R)(v+x_\alpha)-\big[(P-R)(v)+(P-R)x_\alpha\big]\right|=
\left|\sum_{j=1}^{n-1} \binom{n}{j} (P-R)^\vee
(v^j,x_\alpha^{n-j})\right|\underset{\alpha}{\longrightarrow} 0.
$$
This implies the following:
\begin{eqnarray*}
\limsup_\alpha |\lambda+Px_\alpha| &\leq& \limsup_\alpha
|Q(v+x_\alpha)+(P-R)(v+x_\alpha)|+\varepsilon\\
&\leq &(1+\varepsilon)\limsup_\alpha \|v+x_\alpha\|^n+\varepsilon.
\end{eqnarray*}
Being $\varepsilon$ arbitrary, the proof is complete.
\end{proof}

\begin{example}\label{Lp-no-es}\rm
The previous proposition implies that $\mathcal{P}_w(^2L_p[0,1])$ is
not an $M$-ideal in $\mathcal{P}(^2L_p[0,1])$, for $1<p<2$, because
$L_p[0,1]$ does not have the $2$-polynomial property $(M)$. Indeed,
let $\{r_n\}_n$ be the sequence of Rademacher functions and consider
the polynomial $P\in \mathcal{P}(^2L_p[0,1])$ given by
$$P(f)=\sum_n\left(\int fr_n\, d\mu \right)^2.$$ The norm of this polynomial is bounded by $B_{p'}^2$, where $B_{p'}$ is the upper bound on Khintchine inequality (this can be derived from the norm of the usual projection from $L_p[0,1]$ to $\ell_2$ \cite[Proposition 6.4.2]{AK}).
So, the polynomial $\frac{P}{B_{p'}^2}$ has norm smaller than 1 and we see that it does not have the property $(M)$. The
inequality fails when we consider the sequence
$\{r_n\}_n$, which is  weakly null, and the function $g\equiv 1$. Being $P(r_n)=1$, for all $n$, and $\|g\|=1$, we see that
$$
\left|1+\frac{P}{B_{p'}^2}(r_n)\right|=1+\frac{1}{B_{p'}^2}, \qquad\textrm{ while }\qquad
\|g+r_n\|^2=2^{2/p'}, \textrm{ for all }n.
$$
We conclude our argument by proving the inequality:
$$
2^{2/p'}<1+\frac{1}{B_{p'}^2}.
$$
From now on, we denote  $q=p'$. Since $q>2$, we know from \cite{Ha} that:

$$
B_q=\sqrt{2}\left(\frac{\Gamma\left(\frac{q+1}{2}\right)}{\sqrt{\pi}}\right)^{1/q}.
$$ This means that the inequality that we want to prove is the following:
$$
2^{2/q} < 1 +  \frac{1}{2}\cdot\left(\frac{\sqrt \pi}{\Gamma \left(\frac{q+1}{2}\right)}\right)^{2/q}, \textrm{ for every }q>2.
$$

Our first step is to prove an adequate bound for the gamma function:
\begin{equation}\label{gama}
\Gamma\left(\frac{q+1}{2}\right)\leq \left(\frac{q}{4}\right)^{q/2}\sqrt{\pi}, \textrm{ for every }q\geq 2.
\end{equation}

This inequality is equivalent to show the negativity of the function
$$
h(q)=\log\left(\Gamma\left(\frac{q+1}{2}\right)\right)-\frac{q}{2}\log\left(\frac{q}{4}\right) -\log(\sqrt{\pi}),
$$ in the interval $[2,+\infty)$.

The derivative of $h$ is
$$
h'(q)=\frac{1}{2}\left[\psi\left(\frac{q+1}{2}\right)-\log\left(\frac{q}{4}\right)-1\right],
$$
where $\psi(x)=\frac{\Gamma'(x)}{\Gamma(x)}$ is the digamma function, that is the derivative of the logarithm of the gamma function. The function $\psi$ satisfies the following inequality, for all $x>0$, (see, for instance \cite[(2.2)]{Al}):
$$
\psi(x)<\log(x)-\frac{1}{2x}.
$$

Thus, we obtain that
$$
h'(q)< \frac{1}{2}\left[\log\left(\frac{q+1}{2}\right)-\frac{1}{q+1}- \log\left(\frac{q}{4}\right)-1\right]\leq 0.
$$

Consequently, $h$ is decreasing and  $h(q)< h(2)=0$.

This proves the validity  of equation (\ref{gama}), which implies that
$$
1 +  \frac{1}{2}\cdot\left(\frac{\sqrt \pi}{\Gamma \left(\frac{q+1}{2}\right)}\right)^{2/q}\geq 1+\frac{1}{2}\left(\frac{q}{4}\right) = 1+\frac{2}{q}.
$$

The conclusion holds since $1+\frac{2}{q}>2^{2/q}$, for all $q>2$, because $f(q)=\left(1+\frac{2}{q}\right)^{q/2}$ is a strictly increasing function and $f(2)=2$.

\end{example}

\begin{theorem}\label{n-prop-M}
Let $E$ be a Banach space and suppose that there exists a net
$\{K_\alpha\}_\alpha$ of compact operators from $E$ to $E$
satisfying the following two conditions:
\begin{itemize}
\item $K_\alpha x\to x$, for all $x\in E$ and $K_\alpha^* \gamma \to \gamma$, for all $\gamma\in E^*$.
\item $\|Id - 2 K_\alpha\|\underset{\alpha}{\longrightarrow} 1$.
\end{itemize}
Suppose also that $n=cd(E)$. Then,
$\mathcal{P}_w(^nE)$ is an $M$-ideal in $\mathcal{P}(^nE)$ if and
only if $E$ has the $n$-polynomial property $(M)$.
\end{theorem}

\begin{proof}
One direction follows from the previous proposition. For the other,
we will verify the 3-ball property. Let $P_1, P_2, P_3\in
B_{\mathcal{P}_w(^nE)}$, $Q\in B_{\mathcal{P}(^nE)}$ and
$\varepsilon >0$. We will prove that, for $\alpha$ large enough, the
polynomial $P(x)=Q(x)-Q(x-K_\alpha x)$ satisfies that
$\|Q+P_i-P\|\leq 1+\varepsilon$. As in the proof of Proposition
\ref{condition}, it can be seen that $P$ is weakly continuous on
bounded sets.

Let $\beta$ such that
$$
\|Id - 2K_{\beta}\|^n\leq 1+\frac{\varepsilon}{2}\qquad \textrm{ and
} \qquad \|P_i-P_i\circ K_{\beta}\|\leq
\frac{\varepsilon}{2},\quad\textrm{for all }i=1,2,3.
$$
(Recall that from Lemma \ref{Pw}, $\|P_i-P_i\circ K_{\alpha}\|\to
0$). We have that
$$
\|Q+P_i-P\|\leq \|Q+P_i\circ K_\beta-P\| + \|P_i-P_i\circ
K_{\beta}\|\leq \|P_i\circ K_\beta+Q\circ(Id-K_\alpha)\| +
\frac{\varepsilon}{2}.
$$
Let $\{x_\alpha\}_\alpha\subset B_E$ such that
$$
\limsup_\alpha \|P_1\circ K_\beta+Q\circ(Id-K_\alpha)\|=
\limsup_\alpha |P_1(K_\beta x_\alpha)+Q(x_\alpha-K_\alpha
x_\alpha)|.
$$
Since $|P_1(K_\beta x_\alpha)|\leq \|K_\beta x_\alpha\|^n$,
$\{K_\beta x_\alpha\}_\alpha$ is contained in a compact set of $E$,
 $\{x_\alpha -K_\alpha x_\alpha\}_\alpha$ is a bounded weakly
null net and $E$ has the $n$-polynomial property $(M)$, from Lemma
\ref{redes}, we get
\begin{eqnarray*}
\limsup_\alpha |P_1(K_\beta x_\alpha)+Q(x_\alpha-K_\alpha x_\alpha)|
&\leq & \limsup_\alpha \|K_\beta x_\alpha+x_\alpha-K_\alpha
x_\alpha\|^n\\
&\leq & \limsup_\alpha \|K_\beta +Id-K_\alpha\|^n\\
&\leq & \|Id-2K_\beta\|^n \leq 1+\frac{\varepsilon}{2},
\end{eqnarray*}
where the inequality of the last line is proved in \cite[page
300]{HWW}.

Therefore,
$$
\limsup_\alpha \|Q+P_1-(Q-Q\circ(Id-K_\alpha))\|\leq 1+\varepsilon.
$$
So, there exists a subnet $\{K_\gamma\}_\gamma$ of
$\{K_\alpha\}_\alpha$ such that
$$
\lim_\gamma \|Q+P_1-(Q-Q\circ(Id-K_\gamma))\|\leq 1+\varepsilon.
$$
With the same argument, taking further subnets, we obtain the
inequality for $P_2$ and $P_3$. Thus, the 3-ball property is proved.
\end{proof}

The following proposition is the polynomial analogous of \cite[Lemma
VI.4.14]{HWW} (and again we borrow some ideas from that proof) and
enables us to obtain a link with the linear theory.

\begin{proposition}\label{M}
Let $E$ be a Banach space and
$n=cd(E)$. If $E$ has the property
$(M)$, then $E$ has the $n$-polynomial property $(M)$.
\end{proposition}

\begin{proof}
Let $P\in\mathcal{P}(^nE)$ with $\|P\|\leq 1$,  $\lambda\in
\mathbb{K}$, $v\in E$ with $|\lambda|\leq\|v\|^n$ and let
$\{x_\alpha\}_\alpha\subset E$ be a bounded weakly null net. We want to
prove that
$$
\limsup_\alpha |\lambda+Px_\alpha|\leq \limsup_\alpha
\|v+x_\alpha\|^n.
$$
Suppose first that $\|P\|=1$ and that $E$ is a complex Banach space.
Given $\varepsilon >0$, there exists $u_\varepsilon\in E$ such that
$P(u_\varepsilon) =\lambda(1-\varepsilon)$ and
$\|u_\varepsilon\|\leq |\lambda|^{1/n}$ (and so $\|u_\varepsilon\|$
is smaller than $\|v\|$). Thus,
\begin{eqnarray*}
\limsup_\alpha |\lambda+Px_\alpha|&\leq & \limsup_\alpha
|Pu_\varepsilon+Px_\alpha| + |\lambda|\varepsilon \\
&=&\limsup_\alpha |P(u_\varepsilon+x_\alpha)| +
|\lambda|\varepsilon\qquad \textrm{ (since
}\mathcal{P}(^kE)=\mathcal{P}_{w}(^kE),\textrm{ for
} k<n\textrm{)}\\
&\leq & \|P\| \limsup_\alpha \|u_\varepsilon+x_\alpha\|^n +
|\lambda|\varepsilon\\
&\leq & \limsup_\alpha \|v+x_\alpha\|^n +|\lambda|\varepsilon.
\end{eqnarray*}
Since $\varepsilon$ is arbitrary, the inequality is proved.

If $E$ is real, the same argument works, except
in the case that $n$ is even, $Px\geq 0$, for all $x\in E$, and
$\lambda <0$ (or, $Px\leq 0$, for all $x\in E$, and $\lambda >0$).
But, if this is the case, then
$$
|\lambda + Px_\alpha|\leq |-\lambda + Px_\alpha|,
$$
and the above steps prove that $\limsup_\alpha |-\lambda+Px_\alpha|\leq \limsup_\alpha
\|v+x_\alpha\|^n.$

Let us consider now that $\|P\|<1$.  Being $(\lambda +Px_\alpha)$ a
convex combination of $\Big(\lambda +\frac{P}{\|P\|}x_\alpha\Big)$
and $\Big(\lambda -\frac{P}{\|P\|}x_\alpha\Big)$, we get that
$$
|\lambda +Px_\alpha|\leq \max\left\{\left|\lambda
+\frac{P}{\|P\|}x_\alpha\right|,\left|\lambda
-\frac{P}{\|P\|}x_\alpha\right|\right\}.
$$
Therefore,
\begin{eqnarray*}
\limsup_\alpha |\lambda+Px_\alpha|&\leq &
\max\left\{\limsup_\alpha\left|\lambda
+\frac{P}{\|P\|}x_\alpha\right|,\limsup_\alpha\left|\lambda
-\frac{P}{\|P\|}x_\alpha\right|\right\}\\
&\leq & \limsup_\alpha \|v+x_\alpha\|^n.
\end{eqnarray*}
\end{proof}

Now we derive a relationship with the linear theory that enable us to produce more examples of polynomial $M$-structures.

\begin{corollary}
Let $E$ be a Banach space and
$n=cd(E)$. If $\mathcal{K}(E)$ is
an $M$-ideal in $\mathcal{L}(E)$, then $\mathcal{P}_w(^nE)$ is an
$M$-ideal in $\mathcal{P}(^nE)$.
\end{corollary}

\begin{proof}
By \cite[Theorem VI.4.17]{HWW}, if $\mathcal{K}(E)$ is an $M$-ideal
in $\mathcal{L}(E)$, then $E$ has property $(M)$ and there exists a
net $\{K_\alpha\}_\alpha\subset \mathcal{K}(E)$ satisfying $K_\alpha x\to x$, for all $x\in E$,
 $K_\alpha^* \gamma \to \gamma$, for all $\gamma\in E^*$ and
$\|Id - 2 K_\alpha\|\underset{\alpha}{\longrightarrow} 1$. As a
consequence of Theorem \ref{n-prop-M} and Proposition \ref{M}, it
follows that $\mathcal{P}_w(^nE)$ is an $M$-ideal in
$\mathcal{P}(^nE)$.
\end{proof}

\begin{remark}\rm
The reciprocal of the previous corollary does not hold. Indeed, in
Example \ref{Lp}, we see that for $1<p<2$, there is a renorming $E$
of $L_p[0,1]$ such that $\mathcal{P}_w(^2E)$ is an $M$-ideal in
$\mathcal{P}(^2E)$.  But, from \cite[Corollary VI.6.10]{HWW}, we
know that $L_p[0,1]$ can not be renormed to a Banach space $E$ which
makes $\mathcal{K}(E)$  an $M$-ideal in $\mathcal{L}(E)$.
\end{remark}

Note that, by the previous corollary, all the known examples of spaces $E$ such that $\mathcal{K}(E)$ is an $M$-ideal in $\mathcal{L}(E)$ would provide polynomial examples, once we find out what is the critical degree. Recall also that the critical degree is preserved by isomorphism.

\begin{example}\rm
If $\mathbb{D}$ is the complex disc, the Bergman space $B_p$ is the
space of all holomorphic functions in $L_p(\mathbb{D},dxdy)$. If
$1<p<\infty$, $B_p$ is isomorphic to $\ell_p$ \cite[Theorem
III.A.11]{Woj} and so $cd(B_p)=cd(\ell_p)$. By \cite[Corollary 4.8]{KW}, $\mathcal{K}(B_p)$
is an $M$-ideal in $\mathcal{L}(B_p)$. Thus, $\mathcal{P}_w(^nB_p)$
is an $M$-ideal in $\mathcal{P}(^nB_p)$, for $p\leq n<p+1$.
\end{example}

\begin{example}\rm
By \cite[Corollary VI.6.12]{HWW}, an Orlicz sequence space $h_M$ can
be renormed to a space $E$ for which $\mathcal{K}(E)$ is  an
$M$-ideal in $\mathcal{L}(E)$ if and only if $(h_M)^*$ is separable.
Also, $\mathcal{P}_w(^kh_M)=\mathcal{P}(^kh_M)$ if $k<\alpha_M$ and
$\mathcal{P}_w(^kh_M)\not=\mathcal{P}(^kh_M)$ if $k>\beta_M$ (see
\cite{GJ}), where $\alpha_M$ and $\beta_M$ are the lower and upper
Boyd indexes associated to $M$. So, for certain values of $\alpha_M$ and $\beta_M$ the critical degree can be establish. Then, if $(h_M)^*$ is separable and
$n$ is the critical degree,
$\mathcal{P}_w(^nE)$ is  an $M$-ideal in $\mathcal{P}(^nE)$.
\end{example}

\section{Block diagonal polynomials}

In \cite{DG} it is introduced the concept of ``block diagonal
polynomials'' for spaces with unconditional finite dimensional
decomposition and it is studied the relationship between the
equality $\mathcal{P}(^kE)=\mathcal{P}_{w}(^kE)$ and that the same
happens for block diagonal polynomials. We want here to face the
problem  of whether a space of block diagonal polynomials that are
weakly continuous on bounded sets is an $M$-ideal in the space of
all block diagonal polynomials. First, recall the definition:

\begin{definition}
Let $E$ be a Banach space with an unconditional finite dimensional
decomposition with associate projections $\{\pi_m\}_m$ and let
$J=\{m_j\}_j$ be an increasing sequence of positive integers. For
each $j\in\mathbb{N}$, let $\sigma_j=\pi_{m_j}-\pi_{m_{j-1}}$. The
class of {\bf block diagonal $n$-homogeneous polynomials with
respect to $J$} is the set
$$
\mathcal{D}_J(^nE)=\Big\{P\in\mathcal{P}(^nE):
P(x)=\sum_{j=1}^\infty P(\sigma_j(x)),\ \forall x\in E\Big\}.
$$
\end{definition}

Observe that if $J=\mathbb{N}$ and $E$ has an unconditional basis, then $\mathcal{D}_J(^nE)$ is the space of $n$-homogeneous diagonal polynomials. Also, if $E$ is a real Banach space and it has a 1-unconditional basis, diagonal polynomials coincide with orthogonally additive polynomials defined on a Banach lattice.

We want to state conditions that assure that, for a fixed sequence
$J$, the space $\mathcal{D}_{J,w}(^nE)$ is an $M$-ideal in
$\mathcal{D}_{J}(^nE)$. Note that this problem makes sense, because
by \cite[Proposition 1.6]{DG}, for spaces $E$ with shrinking
unconditional finite dimensional decomposition, the existence of a
polynomial which is not weakly continuous on bounded sets implies
the existence of a block diagonal polynomial with respect to some
$J$ which is not weakly continuous on bounded sets. And we have more values of $n$ where to look for $M$-structures (not just a critical degree), because, by
\cite[Proposition 13]{AD},
$\mathcal{D}_{J,w}(^kE)=\mathcal{D}_{J,w0}(^kE)$, for every $k$.

Observe that, if $E$ has unconditional finite dimensional
decomposition with associate projections $\{\pi_m\}_m$ and
$J=\{m_j\}_j$, then for $Q\in\mathcal{D}_J(^nE)$, the polynomial
$$
P(x)=Q(x)-Q(x-\pi_{m_j}x)=Q(\pi_{m_j}x)
$$
is also in $\mathcal{D}_J(^nE)$. So, the proof of Proposition
\ref{condition} easily implies the following proposition.

\begin{proposition}
Let $E$ be a Banach space with a shrinking unconditional finite
dimensional decomposition with associate projections $\{\pi_m\}_m$
and let $J=\{m_j\}_j$ be an increasing sequence of positive
integers. Suppose that for all $\varepsilon >0$ and all
$j_0\in\mathbb{N}$ there exists $j>j_0$ such that for every $x\in
E$,
$$
\|\pi_{m_j} x\|^n + \|x-\pi_{m_j} x\|^n \leq (1+\varepsilon)
\|x\|^n.
$$
Then, for every $k\geq n$, $\mathcal{D}_{J,w}(^kE)$ is an $M$-ideal
in $\mathcal{D}_J(^kE)$.
\end{proposition}

\begin{example}\rm
Let $H$ be a separable Hilbert space. Then, for every
$J\subset\mathbb{N}$ and every $n\geq 2$, $\mathcal{D}_{J,w}(^nH)$
is an $M$-ideal in $\mathcal{D}_J(^nH)$.
\end{example}

\begin{example}\rm
Let $E=\bigoplus_{\ell_p}X_m$, where each $X_m$ is a finite
dimensional space and $1<p<\infty$. Then, for every
$J\subset\mathbb{N}$ and every $n\geq p$, $\mathcal{D}_{J,w}(^nE)$
is an $M$-ideal in $\mathcal{D}_J(^nE)$. Recall that, for $n<p$,
$\mathcal{D}_{J,w}(^nE)=\mathcal{D}_J(^nE)$.

\noindent In particular, we have that $\mathcal{D}_{J,w}(^n\ell_p)$
is an $M$-ideal in $\mathcal{D}_J(^n\ell_p)$, for every $n\geq p$.
\end{example}

\begin{example}\rm
Let $E=d^*(w,p)$, with $1<p<\infty$ and $J\subset\mathbb{N}$. Then,
$\mathcal{D}_{J,w}(^n d^*(w,p))$ is an $M$-ideal in
$\mathcal{D}_J(^n d^*(w,p))$, for every $n\geq p^*$. Recall that
$\mathcal{D}_{J,w}(^n d^*(w,p))=\mathcal{D}_J(^n d^*(w,p))$ for
$n<p^*$ and $w\not\in\ell_s$, where
$s=\Big(\frac{(n-1)^*}{p}\Big)^*$. For $n<p^*$ and $w\in\ell_s$ we
do not know if there is an $M$-ideal structure in this space.
\end{example}

In the case that the unconditional constant of the decomposition
equals 1, it is easy to see that the following holds:
\begin{itemize}
\item $\|\pi_n\|=1$, for every $n\in\mathbb{N}$.

\item $\|Id-2\pi_n\|=1$, for every $n\in\mathbb{N}$.
\end{itemize}

In this situation we have more results about $M$-structure for block
diagonal polynomials. First, a block diagonal version of Proposition
\ref{M-sumando}.

\begin{proposition}
Let $E$ be a Banach space with a 1-unconditional finite dimensional
decomposition with associate projections $\{\pi_m\}_m$ and let
$J=\{m_j\}_j$ be an increasing sequence of positive integers. If
$\mathcal{D}_{J,w}(^nE)$ is an $M$-summand in $\mathcal{D}_J(^nE)$
then $\mathcal{D}_{J,w}(^nE)=\mathcal{D}_J(^nE)$.
\end{proposition}

\begin{proof}
Suppose that $\mathcal{D}_J(^nE)=\mathcal{D}_{J,w}(^nE)
\oplus_\infty S$, where $S\not= \{0\}$. For a given $\varepsilon
>0$, let $P\in S$ and $x_0\in B_E$ such that  $\|P\|=1$ and
$P(x_0)>1-\varepsilon$.

Since $P(x_0)=\sum_{j=1}^\infty P(\sigma_j(x_0))$ and this sum is
absolutely convergent it should exists $N\in \mathbb{N}$ such that
$$
\sum_{j=N+1}^\infty \left|P(\sigma_j(x_0))\right| <\varepsilon.
$$

The polynomial $Q=P\circ \pi_{m_N}$ belongs to
$\mathcal{D}_{J,w}(^nE)$ and $\|Q\|\leq 1$, so $\|Q+P\|=
\max\{\|Q\|, \|P\|\}=1$. But,
$$
(Q+P)(x_0)=\sum_{j=1}^N P(\sigma_j(x_0))+\sum_{j=1}^\infty
P(\sigma_j(x_0))=2P(x_0)-\sum_{j=N+1}^\infty
P(\sigma_j(x_0))>2-3\varepsilon,
$$
which is a contradiction.
\end{proof}

In order to obtain a condition for $\mathcal{D}_{J,w}(^nE)$ to be an
$M$-ideal in $\mathcal{D}_J(^nE)$ involving property $(M)$, we
present the following variation of Proposition \ref{M}.

\begin{proposition}
Let $E$ be a Banach space with a shrinking 1-unconditional finite
dimensional decomposition with associate projections $\{\pi_m\}_m$
and let $J=\{m_j\}_j$ be an increasing sequence of positive
integers. If $E$ has property $(M)$, then the following holds: for
every $P\in\mathcal{D}_J(^nE)$ with $\|P\|\leq 1$, for every
sequence $\{x_k\}_k\in B_E$ and every sequence of scalars
$\{\lambda_k\}_k$ such that $|\lambda_k|\leq \|\pi_{m_{j_0}}x_k\|^n$,
(for a fixed index $j_0$), it follows that
$$
\limsup_k |\lambda_k+P(x_k-\pi_kx_k)|\leq \limsup_k
\|\pi_{m_{j_0}}x_k+x_k-\pi_kx_k\|^n.
$$
\end{proposition}

\begin{proof}
Suppose first that $\|P\|=1$ and that $E$ is a complex Banach space.
Given $\varepsilon>0$, there exists $z_\varepsilon\in B_E$ such that
$P(z_\varepsilon)=1-\varepsilon$.

Let, for every $k\in\mathbb{N}$,
$u_k^\varepsilon=\lambda_k^{1/n}z_\varepsilon$. So,
$P(u_k^\varepsilon)=\lambda_k(1-\varepsilon)$ and
$\|u_k^\varepsilon\|\leq \|\pi_{m_{j_0}}x_k\|$. Thus,
$$
\limsup_k |\lambda_k+P(x_k-\pi_kx_k)|\leq  \limsup_k
|P(u_k^\varepsilon)+P(x_k-\pi_kx_k)| + |\lambda_k|\varepsilon.
$$
Since the series $\sum_j P(\sigma_j(z_\varepsilon))$ converges
absolutely, there exists $N\in\mathbb{N}$ such that
$\sum_{j=N+1}^\infty |P(\sigma_j(z_\varepsilon))|<  \varepsilon$.
Thus,
$$
\left|P(u_k^\varepsilon)-P(\pi_{m_N}u_k^\varepsilon)\right|<
|\lambda_k|\varepsilon\leq \varepsilon,
$$
because $|\lambda_k|\leq 1$.

This implies that
\begin{eqnarray*}
\limsup_k |\lambda_k+P(x_k-\pi_kx_k)|&\leq & \limsup_k
|P(\pi_{m_N}u_k^\varepsilon)+P(x_k-\pi_kx_k)| + 2\varepsilon\\
&=& \limsup_k
|P(\pi_{m_N}u_k^\varepsilon + x_k-\pi_kx_k)| + 2\varepsilon\\
&\leq & \limsup_k \|\pi_{m_N}u_k^\varepsilon + x_k-\pi_kx_k\|^n + 2\varepsilon\\
&\leq & \limsup_k \|\pi_{m_{j_0}}x_k + x_k-\pi_kx_k\|^n +
2\varepsilon,
\end{eqnarray*}
where the last step follows from property $(M)$ since the sequence
$\{x_k-\pi_kx_k\}_k$ is weakly null and
$\|\pi_{m_N}u_k^\varepsilon\|\leq \|u_k^\varepsilon\|\leq
\|\pi_{m_{j_0}}x_k\|$.

Since this happens for arbitrary $\varepsilon$, the result follows. If $E$
is real or $\|P\|<1$, we can repeat the argument of the proof of
Proposition \ref{M} to obtain the desired result.
\end{proof}

\begin{theorem}
Let $E$ be a Banach space with a shrinking 1-unconditional finite
dimensional decomposition with associate projections $\{\pi_m\}_m$
and let $J=\{m_j\}_j$ be an increasing sequence of positive
integers. If $E$ has property $(M)$, then, for every
$n\in\mathbb{N}$, $\mathcal{D}_{J,w}(^nE)$ is an $M$-ideal in
$\mathcal{D}_J(^nE)$.
\end{theorem}

\begin{proof}
Let us see that the 3-ball property is satisfied. Let
$P_1,P_2,P_3\in B_{\mathcal{D}_{J,w}(^nE)}$, $Q\in
B_{\mathcal{D}_{J}(^nE)}$, $\varepsilon>0$. We want to show that,
for $j$ large enough, the polynomial
$P=Q\circ\pi_{m_j}\in\mathcal{D}_{J,w}(^nE)$ satisfies
$\|Q+P_i-P\|\leq 1+\varepsilon$, for $i=1,2,3$.

Since $P_i-P_i\circ\pi_m \underset{m\to\infty}{\longrightarrow}0$,
there is $j_0\in\mathbb{N}$ such that
$$
\|P_i-P_i\circ\pi_{m_{j_0}}\|\leq \varepsilon, \quad \textrm{for
}i=1,2,3.
$$

Now we have,
$$
\|Q+P_i-P\|\leq \|Q+P_i\circ\pi_{m_{j_0}}-P\|+
\|P_i-P_i\circ\pi_{m_{j_0}}\| \leq
\|Q+P_i\circ\pi_{m_{j_0}}-Q\circ\pi_{m_j}\|+\varepsilon.
$$

Let $\{x_j\}_j\subset B_E$ such that

$$
\limsup_j \|Q+P_i\circ\pi_{m_{j_0}}-Q\circ\pi_{m_j}\|=\limsup_j
|Q(x_j)+(P_i\circ\pi_{m_{j_0}})(x_j)-(Q\circ\pi_{m_j})(x_j)|.
$$

Since the polynomial $Q$ is block diagonal with respect to $J$, it
holds that $Q(x_j)-Q(\pi_{m_j}(x_j))= Q(x_j-\pi_{m_j}(x_j))$. From
this and the previous proposition we have,

\begin{eqnarray*}
\limsup_j \|Q+P_i\circ\pi_{m_{j_0}}-Q\circ\pi_{m_j}\| & = &
\limsup_j |P_i(\pi_{m_{j_0}}(x_j))-Q(x_j-\pi_{m_j}(x_j))|\\
&\leq & \limsup_j \|\pi_{m_{j_0}}(x_j)+ x_j-\pi_{m_j}(x_j)\|^n\\
&\leq & \limsup_j \|\pi_{m_{j_0}}+ Id-\pi_{m_j}\|^n =1.
\end{eqnarray*}

With final considerations as in the proof of Theorem \ref{n-prop-M},
the result is proved.

\end{proof}

\section*{Acknowledgements}

The author wishes to thank Silvia Lassalle and Daniel Carando for
helpful conversations and suggestions.

\end{document}